\documentclass[12pt]{article}
\usepackage[latin1]{inputenc}
\usepackage[british]{babel}
\usepackage{lmodern}
\usepackage[T1]{fontenc}
\usepackage[paper=a4paper, left=28mm, right=25mm, top=29mm, bottom=25mm]{geometry}
\usepackage{latexsym,amsfonts,amsmath,graphics}
\usepackage{epsfig}
\usepackage{float}
\usepackage{enumitem}
\usepackage{amssymb,mathrsfs}
\usepackage{verbatim}

\DeclareMathOperator*{\argmin}{arg\,min}

\newtheorem{theorem}{Theorem}
\newtheorem{lemma}{Lemma}
\newtheorem{corollary}{Corollary}
\newtheorem{proposition}{Proposition}
\newtheorem{conjecture}{Conjecture}
\newtheorem{remark}{Remark}

\newtheorem{algorithm}{Algorithm}
\newenvironment{proof}{\begin{trivlist} \item[\hskip\labelsep{\it Proof.}]}{$\hfill\Box$\end{trivlist}}

\newcommand{\rd}{\,\mathrm{d}}

\newcommand{\bsx}{\boldsymbol{x}}
\newcommand{\bsy}{\boldsymbol{y}}

\newcommand{\bst}{\boldsymbol{t}}

\newcommand{\bsgamma}{\boldsymbol{\gamma}}

\newcommand{\bszero}{\boldsymbol{0}}

\newcommand{\RR}{\mathbb{R}}

\newcommand{\NN}{\mathbb{N}}

\newcommand{\cV}{\mathcal{V}}

\newcommand{\cP}{\mathcal{P}}
\newcommand{\cS}{\mathcal{S}}

\allowdisplaybreaks

\title{Uniformly distributed sequences generated by a greedy minimization of the $L_2$ discrepancy}
\author{Ralph Kritzinger 
}
\date{}

\begin{document}

\maketitle

\begin{abstract}
The aim of this paper is to develop greedy algorithms which generate uniformly distributed sequences in the $d$-dimensional unit cube $[0,1]^d$. The figures of merit are three different variants of $L_2$ discrepancy. Theoretical results along with numerical experiments suggest that the resulting sequences have excellent distribution properties. The approach we follow here is motivated by recent work of Steinerberger and Pausinger who consider similar greedy algorithms, where they minimize functionals that can be related to the star discrepancy or energy of point sets. In contrast to many greedy algorithms where the resulting elements of the sequence can only be given numerically, we will find that in the one-dimensional case our algorithms yield rational numbers which we can describe precisely. In particular, we will observe that any initial segment of a sequence in $[0,1)$ can be naturally extended to a uniformly distributed sequence where all subsequent elements are of the form $x_N=\frac{2n-1}{2N}$ for some $n\in\{1,\dots,N\}$. We will also investigate the dependence of the $L_2$ discrepancy of the resulting sequences on the dimension $d$.
\end{abstract}

\centerline{\begin{minipage}[hc]{130mm}{
}
\end{minipage}}
 
 \allowdisplaybreaks
 
\section{Introduction}

The theory of uniform distribution modulo 1 dates back to the seminal work of Weyl~\cite{Weyl} in 1916 and is concerned with the following elementary question: How does one construct a sequence $(x_n)_{n\geq 1}$ of real numbers in the unit interval $[0,1)$ such that for all numbers $N\geq 1$ and all intervals $J\subseteq [0,1]$ the difference
$$ |\#\{1\leq n\leq N: x_n\in J\}-N|J|| $$
is small or even bounded in $N$? For a fixed number $N$, it is not difficult to see that
\begin{equation*}  \sup_{J\subseteq [0,1]}|\#\{1\leq n\leq N: x_n\in J\}-N|J||\lesssim 1 \end{equation*}
($J$ is always an interval) if we choose an equispaced set of points $x_1\leq x_2\leq\dots\leq x_N$; i.e. $$|x_{n+1}-x_n|=\frac{1}{N} \text{\, for all \,} n=2,\dots,N.$$ However, as we are asked to build an infinite sequence, the number $N$ increases and we cannot use equispaced point sets. In fact it is known from a celebrated result by Schmidt~\cite{Schm72distrib} that any sequence $(x_n)_{n\geq 1}$ in the unit interval satisfies
\begin{equation}
    \label{Schmidt3}\sup_{J\subseteq [0,1]}|\#\{1\leq n\leq N: x_n\in J\}-N|J||\gtrsim  \log{N}
\end{equation}  
for infinitely many $N$. Schmidt's result implies that it is not possible to construct a sequence in the unit interval which is distributed too regularly for all $N$ simultaneously. Therefore, one speaks of irregularities of distribution. \\
It is known that Schmidt's lower bound is sharp. Two classical and well-known constructions which attain the $(\log{N})$-rate in~\eqref{Schmidt3} are:
\begin{itemize}
    \item The $(n\alpha)$-sequence or Kronecker sequence, where $x_n=n\alpha \pmod{1}$ for an irrational number $\alpha$ with a bounded continued fraction expansion.
    \item The van-der-Corput sequence as defined in Section~\ref{defisec}.
\end{itemize}
Analogue questions can be asked for sequences in the $d$-dimensional unit cube $[0,1]^d$ too. By a famous result of Roth in 1954 it is known that for any sequence $(\bsx_n)_{n\geq 0}$ in $[0,1]^d$ we have
\begin{equation} \label{Roth} \sup_{J\subseteq [0,1]^d}|\#\{1\leq n\leq N: \bsx_n\in J\}-N|J||\gtrsim (\log{N})^{\frac{d}{2}} \end{equation}
for infinitely many $N$. Roth's lower bound has since then been improved by Beck~\cite{beck} in dimension 2 and  Bilyk, Lacey and Vagharshakyan~\cite{bilyk} for $d\geq 2$, but the correct exponent of the $(\log{N})$-term is still unknown. The best known constructions of infinite sequences $(\bsx_n)_{n\geq 0}$ satisfy upper bounds of the form 
\begin{equation} \label{Schmidt} \sup_{J\subseteq [0,1]^d}|\#\{1\leq n\leq N: \bsx_n\in J\}-N|J||\lesssim (\log{N})^{d}, \end{equation}
where all these constructions are based on number-theoretic methods like digit expansions. Many researchers in the theory of uniform distribution believe that the exponent $d$ is best possible, while others conjecture that the exponent $\frac{d}{2}$ in Roth's lower bound should be replaced by $\frac{d+1}{2}$ to get a sharp lower bound. Both conjectures are in accordance with Schmidt's result for $d=1$. If we assume the exact exponent to be $\frac{d+1}{2}$ rather than $d$ (or at least some number smaller than $d$), then we need new constructions of extremely well distributed sequences in $[0,1]^d$. \\
Since the known algorithms based on digit expansions are not strong enough to yield such sequences, Steinerberger~\cite{Stein1,Stein2} recently proposed a novel method to construct sequences dynamically: Given an initial set $\{x_1,\dots,x_{N-1}\}$ of points in $[0,1)$, he constructs $x_{N}$ in a greedy manner
  \begin{equation} \label{greedy1} x_{N}=\argmin_{x\in [0,1]}\sum_{n=1}^{N-1} f(|x-x_n|) \end{equation}
  for suitable functions $f: [0,1]\to \mathbb{R}$.
  Hence, every subsequent element of the sequence is chosen dynamically such that a certain functional which depends on the already constructed points is minimized. Steinerberger proved that the resulting sequences are uniformly distributed and he performed numerical calculations which suggest that the sequences are very well distributed in the unit interval. He showed analogous results on higher dimensional sequences as well. Pausinger~\cite{Paus} recently proved that under certain conditions on the function $f$ in~\eqref{greedy1} the resulting sequences are generalized van-der-Corput sequences, which demonstrates that the dynamically defined sequences have the potential to be very well distributed indeed. \\
  We will adopt Steinerberger's and Pausinger's ideas and construct sequences in a greedy manner. We will present here one of our new results in dimension 1. We construct a sequence such that
   $$ x_1=\frac12; \qquad x_{N+1}:=\argmin_{x\in [0,1)} \left(-2\sum_{n=1}^N \min\{x_n,x\}+(N+1)x^2-x\right).  $$
  The motivation behind this construction is based on an $L_2$ approach and will be explained in Sections~\ref{sec2} and~\ref{sec3}. The algorithm generates a sequence starting with the points
    \begin{equation} \label{tollefolge}
        \frac12,\frac14,\frac56,\frac18,\frac{7}{10},\frac{5}{12},\frac{13}{14},\frac{1}{16},\frac{11}{18},\frac{7}{20},\dots
    \end{equation}
    We will prove that the elements of this sequence are of the form
    $$ x_N=\frac{2n-1}{2N} \text{\, for some \,} n=1,\dots,N\text{\, for all \,} N\geq 1.$$
    More generally, we find that independently of the points $\{x_1,\dots,x_k\}$ we start the algorithm with the subsequent elements of the sequence are all rational numbers of the above form. We will observe numerically that for this particular sequence the quantity
    $$\sup_{x\in [0,1]}|\#\{1\leq n\leq N: x_n\in [0,x)\}-Nx|$$
    is very small for all $N$ and seems to be even better behaved than for the van-der-Corput sequence. We view these results as an indication that the greedy approach is very effective. \\
    We will prove that our sequence~\eqref{tollefolge} satisfies
    $$ \left(\int_0^1 |\#\{1\leq n\leq N: x_n\in [0,x)\}-Nx|^2\rd x\right)^{\frac12} \lesssim \sqrt{N},$$
    which implies uniform distribution of the sequence in particular. We will further establish the same result for higher-dimensional sequences generated by new greedy algorithms. \\
    Summarizing, the sequence~\eqref{tollefolge} seems to have remarkable properties. It seems to improve upon classical constructions such as the van-der-Corput sequence with respect to uniform distribution and the established bound $\sqrt{N}$ seems far from the truth, which is probably $\sqrt{\log{N}}$. In order to prove the latter bound (if it is true), it is necessary to understand why the sequence~\eqref{tollefolge} works so well. Perhaps it is possible to find an explicit formula for the elements of the sequence or a simpler recursion, which enables a thorough analysis of the distribution properties of the sequence. \\
    The rest of this paper will be structured as follows. In Section~\ref{defisec} we will give the formal definitions of the quantities of interest and state known results which are relevant in the context of the paper. In Section~\ref{sec2}, we explain how our greedy algorithms work. We construct sequences in $[0,1)^d$ such that the inclusion of a new point $\bsx_{N+1}$ always minimizes the $L_2$ discrepancy of the initial segment $\{\bsx_1,\dots,\bsx_N,\bsx_{N+1}\}$, where $\bsx_1,\dots,\bsx_N$ are already constructed. 
We will show that our greedy algorithms yield sequences that are uniformly distributed modulo 1. We analyze the one-dimensional case in more detail in Section~\ref{sec3}, where we are able to prove precise results on the structure of the resulting sequences. We investigate the dependence of the $L_2$ discrepancy of our sequences on the dimension $d$ in Section~\ref{tract} and propose another greedy algorithm based on the weighted $L_2$ discrepancy. We close with a conclusion in Section~\ref{sec5}.
    
\section{Definitions and background} \label{defisec}

Now we give the necessary definitions and state results which are important in the context of the paper.

We consider an infinite sequence $\cS=\{\bsx_1,\bsx_2,\bsx_3,\dots\}$ of points in the $d$-dimensional unit cube. Let $B$ be a measurable subset of $[0,1]^d$. For a uniformly distributed sequence the number of points among the first $N$ elements of $\cS$ which lie in $B$ should be close to the Lebesgue measure $|B|$ of $B$, multiplied with the total number of points. This motivates the definition of the local discrepancy of the first $N$ elements of $\cS$ with respect to $B$ by
$$ \Delta_N(B,\cS):=\#\{n\in\{1,\dots,N\}: \bsx_n\in B\}-N|B|. $$
In the introduction we considered the supremum of the absolute value of the local discrepancy over all intervals $J\subseteq [0,1]^d$. In literature, this quantity is often called the (star) discrepancy of the sequence $\cS$, which is formally given by
$$  D_N(\cS):=\sup_{J \subset [0,1]^d} |\Delta_N(J,\cS)|. $$
It is well known that is suffices to only consider intervals $J$ which are anchored in the origin; i.e. intervals of the form $J=[\bszero,\bst)$,
where for $\bst=(t_1,\dots,t_d)\in [0,1]^d$ we set $[\bszero,\bst)=[0,t_1)\times [0,t_2)\times\dots\times [0,t_d)$. In this case, we write $D_N^*$ instead of $D_N$. \\
Clearly, it makes sense to measure the average behaviour of the local discrepancy by calculating its $L_2$ norm. The corresponding quantity is called (star) $L_2$ discrepancy and is formally defined as
$$ L_{2,N}(\cS)=L_{2,N}^{\mathrm{star}}(\cS):=\left(\int_{[0,1]^d} |\Delta_N([\bszero,\bst),\cS)|^2 \rd\bst\right)^{\frac{1}{2}}. $$
In the case of $L_2$ discrepancy, it turns out that (in constrast to the case of supremum norm) we obtain two non-equivalent measures of distribution of sequences if we choose arbitrary subintervals $J\subseteq [0,1]^d$ as test sets instead of intervals anchored in the origin. We therefore introduce the extreme $L_2$ discrepancy, where we use arbitrary subintervals $[\bsx,\bsy)$ of $[0,1]^d$ as test sets. Here for $\bsx=(x_1,\dots,x_d)\in [0,1]^d$ and $\bsy=(y_1,\dots,y_d)\in [0,1]^d$ with $\bsx\leq\bsy$, i.e. $x_i\leq y_i$ for all $i\in \{1,2,\dots,d\}$, we define $[\bsx,\bsy)=[x_1,y_1)\times\dots\times [x_d,y_d)$. The extreme $L_2$ discrepancy is then defined as
$$ L_{2,N}^{\mathrm{extr}}(\cS):=\left(\int_{[0,1]^d}\int_{[0,1]^d,\, \bsx\leq \bsy} |\Delta_N([\bsx,\bsy),\cS)|^2 \rd\bsx\rd\bsy\right)^{\frac{1}{2}}. $$
The periodic $L_2$ discrepancy uses periodic boxes as test sets, which can be introduced as follows: For $x,y\in [0,1]$ set
$$ I(x,y)=\begin{cases}
           [x,y) & \text{if $x\leq y$}, \\
           [0,y)\cup [x,1)& \text{if $x>y$,}
          \end{cases}$$
and for $\bsx,\bsy$ as above define $B(\bsx,\bsy)=I(x_1,y_1)\times\dots\times I(x_d,y_d)$.
We define the periodic $L_2$ discrepancy of $\cS$ as
$$  L_{2,N}^{\mathrm{per}}(\cS):=\left(\int_{[0,1]^d}\int_{[0,1]^d}|\Delta_N(B(\bsx,\bsy),\cS)|^2\rd \bsx\rd\bsy\right)^{\frac{1}{2}}.  $$
It is known that in dimension 1 we have the relation $L_{2,N}^{\mathrm{per}}(\cS)^2=2L_{2,N}^{\mathrm{extr}}(\cS)^2$ for every sequence $\cS\subset [0,1]$ (see~\cite[Theorem 7]{HKP20}). Further, the periodic $L_2$ discrepancy divided by $N$ is (up to a
multiplicative factor) exactly the diaphony. The diaphony is a well-known measure for
the irregularity of the distribution of point sets and was introduced by Zinterhof~\cite{zint}.

A sequence $\cS=\{\bsx_1,\bsx_2,\bsx_3,\dots\}\subset [0,1)^d$ is called uniformly distributed modulo 1 if and only if
$$ \lim_{N\to\infty}\frac{\#\{n\in\{1,\dots,N\}: \bsx_n\in [\bsx,\bsy)\}}{N}=|[\bsx,\bsy)| $$
for all intervals $[\bsx,\bsy)\subseteq [0,1]^d$. It is well-known that a sequence $\cS$ is uniformly distributed if and only if $$\lim_{N\to\infty}N^{-1} L_{2,N}^{\bullet}(\cS)=0 \text{\, for any \,} \bullet\in \{\mathrm{star},\mathrm{extr},\mathrm{per}\}.$$ For the star $L_2$ discrepancy see e.g.~\cite[Theorem 1.6, Theorem 1.8]{DT97}). For the extreme $L_2$ discrepancy we refer to~\cite{moro} and for the periodic $L_2$ discrepancy or diaphony consult~\cite{zint}. \\
The $L_2$ discrepancy of a sequence cannot be arbitrarily small. For any sequence $\cS$ in $[0,1]^d$ we have
$$L_{2,N}^{\bullet}(\cS)\gtrsim (\log{N})^{\frac{d}{2}} \text{\, for infinitely many \,} N,$$ where $\bullet\in \{\mathrm{star},\mathrm{extr},\mathrm{per}\}$ and where the implied constant depends only on $d$. We refer to~\cite{pro86} and~\cite[Theorem 3, Theorem 5]{kirk} for the star and periodic $L_2$ discrepancy. The claim for the extreme $L_2$ discrepancy can be found in the recent note~\cite[Theorem 1]{KP21}, from which the lower bounds on the star and periodic $L_2$ discrepancy follow as well, which both dominate the extreme $L_2$ discrepancy~\cite[Equ. (1) and Theorem 6]{HKP20}. 
For $\bullet\in\{\mathrm{star},\mathrm{extr}\}$ it is known that there exist sequences which match these bounds, e.g. higher order digital sequences (see~\cite{DP14neu}) or Halton sequences for $d\geq 2$ (see~\cite{levin}). \\
We survey the case $d=1$ in more detail. The van der Corput sequence~\cite{vdc1,vdc2} is a classical example of a so-called low-discrepancy sequence. It is defined as follows: If $n\in\NN_0$ has the dyadic expansion $n=\sum_{j=0}^{m}n_j2^j$ for some $m\in\NN_0$ and digits $n_j\in\{0,1\}$ for $j=0,\dots,m$, set $\varphi(n):=\sum_{j=0}^m n_j2^{-j-1}$. Then the van der Corput sequence (in base $2$) is defined as  $\cV:=(\varphi(n))_{n\in\NN_0}$. It is known (see~\cite{chafa,proat}) that
$$ \limsup_{N\to\infty}\frac{L_{2,N}(\cV)}{\log{N}}=\frac{1}{6\log{2}}, $$
i.e. $L_{2,N}(\cV)=\mathcal{O}(\log{N})$, which is not best possible in $N$. However, a simple modification of the van der Corput sequence matches the optimal $L_2$ discrepancy bound. For the symmetrized van der Corput sequence
$$ \widetilde{\cV}:=\{\varphi(0),1-\varphi(0),\varphi(1),1-\varphi(1),\varphi(2),1-\varphi(2),\dots\} $$
we have $L_{2,N}( \widetilde{\cV})\lesssim \sqrt{\log{N}}$ for all $N\in\NN$ (see~\cite{fau90,lp,pro1988a}). Since the normalized periodic $L_2$ discrepancy is up to multiplicative constants the diaphony, it holds $L_{2,N}^{\mathrm{per}}(\cV)=\sqrt{2}L_{2,N}^{\mathrm{extr}}(\cV)\lesssim \sqrt{\log{N}}$ for all $N\geq 1$ (see~\cite{chafa}). A proof of the optimal extreme $L_2$ discrepancy bound of the symmetrized van der Corput sequence is also possible via Haar functions; see~\cite[Remark 39]{KP21}. Hence, the extreme and periodic $L_2$ discrepancy do not require a symmetrization of the van der Corput sequence in order to achieve the best possible order in $N$. 
\\
As already mentioned in the introduction, the van-der-Corput sequence attains the best possible order of star discrepancy in $N$. More precisely, it is known (see~\cite{befa}) that
$$\limsup_{N \rightarrow \infty} \frac{D_N^*(\cV)}{\log N}=\frac{1}{3 \log 2}. $$

\section{Greedy algorithms and general upper bounds on the $L_2$ discrepancy} \label{sec2}

In this section we will explain our new greedy algorithms and establish the general $L_2$ discrepancy bound of order $\sqrt{N}$, which we announced in the introduction. \\
The idea is the following. We construct a sequence $\cS$ element by element such that the inclusion of the next element of $\cS$ always minimizes the $L_2$ discrepancy. In the following, for sets $X$, $Y$ (the latter totally ordered) and $M\subseteq X$ and a function $f: X\to Y$ we set $$\argmin_{x\in M} f(x):=\{x\in M: f(x)\leq f(y) \text{\, for all \,} y\in M\}.$$  By $x^*\in \argmin_{x\in M} f(x)$ we express that $x^*$ may be chosen as any number in the set $\argmin_{x\in M}f(x)$.

\begin{algorithm}
  For $\bullet\in\{\mathrm{star},\mathrm{extr},\mathrm{per}\}$ let $\cS_d^{\bullet}=\{\bsx_1,\bsx_2,...\}\subset [0,1)^d$ be generated as follows:
\begin{enumerate} \label{mainalgo}
    \item For some arbitrary integer $k\geq 1$ choose $\cP_k=\{\bsx_1,\dots,\bsx_k\}\subset [0,1)^d$ arbitrarily.
    \item For $N\geq k$ let $\{\bsx_1,\dots,\bsx_{N}\}$ already be given. Choose
         \begin{equation} \label{argmin}
             \bsx_{N+1}\in\argmin_{\bsy\in[0,1)^d}L_{2,N+1}^{\bullet}(\{\bsx_1,\dots,\bsx_{N},\bsy\}).
         \end{equation}
\end{enumerate} 
\end{algorithm}

Now we prove the following upper bound on the $L_2$ discrepancy.

\begin{theorem} \label{maintheo}
  For $\bullet\in\{\mathrm{star},\mathrm{extr},\mathrm{per}\}$ a sequence $\cS_d^{\bullet}$ generated by Algorithm~\ref{mainalgo} satisfies
    \begin{equation} \label{mainbound} L_{2,N}(\cS_d^*)^2\leq \max\{c_d^{\bullet},L_{2,k}^{\bullet}(\cP_k)^2\}\cdot(N-k+1), \end{equation}
    for all  $N\geq k$, where
    $$ c_d^{\mathrm{star}}=c_d^{\mathrm{per}}:=\frac{1}{2^d}-\frac{1}{3^d} \text{\,\, and \,\,} c_d^{\mathrm{extr}}:=\frac{1}{6^d}-\frac{1}{12^d}. $$
    In particular, the sequence $\cS_d^{\bullet}$ is uniformly distributed modulo 1.
\end{theorem}

\begin{proof}
 We only show the result for $\bullet=\mathrm{star}$, since the proof of the other two cases follows similar lines, respectively. \\
 Let $\cS=\{\bsx_1,\bsx_2,\dots\}$ be a sequence in $[0,1)^d$, where we write $\bsx_n=(x_{n,1},\dots,x_{n,d})$ for $n\in\NN$. Then we have
 \begin{align} \label{warn1}
  L_{2,N}(\cS)^2=\frac{N^2}{3^d}-\frac{N}{2^{d-1}}\sum_{n=1}^{N}\prod_{i=1}^d (1-x_{n,i}^2)+\sum_{n,m=1}^{N}\prod_{i=1}^d \left(1-\max\{x_{n,i},x_{m,i}\}\right).
\end{align}
This formula follows by simple integration from the definition of the star $L_2$ discrepancy and can also be found in~\cite{warn}. It is often referred to as Warnock's formula. Using~\eqref{warn1} for $ L_{2,N+1}(\cS)^2$, we can derive the recursive formula
\begin{align} \label{recursion}
   L_{2,N+1}(\cS)^2=& \nonumber L_{2,N}(\cS)^2+\frac{2N+1}{3^d}-\frac{1}{2^{d-1}}\sum_{n=1}^{N}\prod_{i=1}^d (1-x_{n,i}^2) -\frac{N+1}{2^{d-1}}\prod_{i=1}^d (1-x_{N+1,i}^2)\\ &+ 2\sum_{n=1}^{N}\prod_{i=1}^d \left(1-\max\{x_{n,i},x_{N+1,i}\}\right)+\prod_{i=1}^d (1-x_{N+1,i}),
\end{align}
which holds for all $N\geq 1$.  The first expression in the right-hand-side of~\eqref{warn1} for $ L_{2,N+1}(\cS)^2$ can be written as
   \begin{align*}
       \frac{(N+1)^2}{3^d}=\frac{N^2}{3^d}+\frac{2N+1}{3^d}.
   \end{align*}  
   Further we have
   \begin{align*}
       &-\frac{N+1}{2^{d-1}}\sum_{n=1}^{N+1}\prod_{i=1}^d (1-x_{n,i}^2) \\=& -\frac{N}{2^{d-1}}\sum_{n=1}^{N}\prod_{i=1}^d (1-x_{n,i}^2)-\frac{1}{2^{d-1}}\sum_{n=1}^{N}\prod_{i=1}^d (1-x_{n,i}^2)-\frac{N+1}{2^{d-1}}\prod_{i=1}^d (1-x_{N+1,i}^2)
   \end{align*}
   and 
   \begin{align*}
      & \sum_{n,m=1}^{N+1}\prod_{i=1}^d \left(1-\max\{x_{n,i},x_{m,i}\}\right) \\ =&
      \sum_{n,m=1}^{N}\prod_{i=1}^d \left(1-\max\{x_{n,i},x_{m,i}\}\right)+ 2\sum_{n=1}^{N}\prod_{i=1}^d \left(1-\max\{x_{n,i},x_{N+1,i}\}\right)+\prod_{i=1}^d (1-x_{N+1,i}).
   \end{align*}
   Inserting these expressions into Warnock's formula~\eqref{warn1} yields~\eqref{recursion}. \\
   Now we prove~\eqref{mainbound} for $\bullet=\mathrm{star}$. 
The assertion is trivially true for $N=k$. Assume it is true for some fixed $N\geq k$. For $\bsy=(y_1,\dots,y_d)\in [0,1]^d$
define \begin{align*}
   \mathcal{L}_N(\bsy):=& L_{2,N}(\cS_d^*)^2+\frac{2N+1}{3^d}-\frac{1}{2^{d-1}}\sum_{n=1}^{N}\prod_{i=1}^d (1-x_{n,i}^2)-\frac{N+1}{2^{d-1}}\sum_{n=1}^{N}\prod_{i=1}^d (1-y_i^2) \\
   &+ 2\sum_{n=1}^{N}\prod_{i=1}^d \left(1-\max\{x_{n,i},y_i\}\right)+\prod_{i=1}^d (1-y_i).
\end{align*}
By~\eqref{recursion} we get that $\mathcal{L}_N(\bsy)\geq 0$ for all $\bsy\in [0,1]^d$. By definition of the sequence $\cS_d^{\mathrm{star}}$ and the induction hypothesis we deduce
\begin{align*}
    L_{2,N+1}(\cS_d^{\mathrm{star}})^2=& \min_{\bsy\in [0,1]^d} \mathcal{L}_N(\bsy) \leq \int_{[0,1]^d} \mathcal{L}_N(\bsy) \rd \bsy\\ =&L_{2,N}(\cS_d^{\mathrm{star}})^2+\frac{1}{2^d}-\frac{1}{3^d}\leq \max\{c_d^{\mathrm{star}},L_{2,k}(\cP_k)^2\} \cdot(N-k+2).
\end{align*}
We used $\int_0^1 (1-y_i^2)\rd y_i=\frac23$, $\int_0^1 (1-y_i)\rd y_i=\frac12$ and $\int_0^1 \left(1-\max\{x_{n,i},y_i\}\right)\rd y_i=\frac12\left(1-x_{n,i}^2\right)$ for all $i\in \{1,\dots,d\}$ in the second to last step. The proof is complete.
\end{proof}

Warnock type formulas exist for the extreme and periodic $L_2$ discrepancy too (see e.g.~\cite{HKP20,HOe,moro,NW10}), which can be used to show Theorem~\ref{maintheo} for $\bullet\in\{\mathrm{extr},\mathrm{per}\}$. As a side result from the proof of Theorem~\ref{maintheo}, we obtain the following relations:
\begin{align*}
    \argmin_{\bsy\in[0,1)^d}L_{2,N+1}^{\bullet}(\{\bsx_1,\dots,\bsx_{N},\bsy\})=\argmin_{\bsy\in[0,1)^d} f_{N,d}^{\bullet},
\end{align*}
where for $\bsy=(y_1,\dots,y_d)\in [0,1)^d$ we define $f_{N,d}^{\bullet}: [0,1)^d \to\RR$ such that
\begin{align*}
    f_{N,d}^{\mathrm{star}}&:=-\frac{N+1}{2^{d-1}}\prod_{i=1}^d (1-y_i^2) + 2\sum_{n=1}^{N}\prod_{i=1}^d \left(1-\max\{x_{n,i},y_i\}\right)+\prod_{i=1}^d (1-y_i), \\
    f_{N,d}^{\mathrm{extr}}&:=\left(1-\frac{N+1}{2^{d-1}}\right)\prod_{i=1}^d y_i(1-y_i) + 2\sum_{n=1}^{N}\prod_{i=1}^d \left(\min\{x_{n,i},y_i\}-x_{n,i}y_i\right), \\
    f_{N,d}^{\mathrm{per}}&:=\sum_{n=1}^{N}\prod_{i=1}^d \left(\frac12-|x_{n,i}-y_i|+(x_{n,i}-y_i)^2\right).
\end{align*}

\section{Further results on the one-dimensional case} \label{sec3}
\subsection{Greedy minimization of the star $L_2$ discrepancy}
Let $\cP=\{y_1,\dots,y_N\}$ be a point set in $[0,1)$, where $y_1\leq y_2 \leq ... \leq y_N$. Then the $L_2$ discrepancy of $\cP$ is given by
\begin{equation}
    L_{2,N}(\cP)^2=N\sum_{n=1}^{N}\left(y_n-\frac{2n-1}{2N}\right)^2+\frac{1}{12}. \label{minimal}
\end{equation}  
This formula can be derived directly from~\eqref{warn1} for $d=1$ and can also be found in~\cite[Corollary 1.1]{nied}.
We immediately conclude that for a fixed $N$ the unique $N$-element point set in $[0,1)$ with minimal $L_2$ discrepancy is the centred regular grid; i.e. the point set
$$ \Gamma_{N}:=\left\{\frac{2n-1}{2N}: n=1,2,\dots,N\right\}. $$
Therefore, for a fixed number $N$ of points the best point distribution with respect to $L_2$ discrepancy is known. However, constructing an infinite sequence such that the segment of the first $N$
 elements does have low discrepancy for all $N\geq 1$ is a much more difficult problem. Therefore we utilize the greedy algorithm from the previous chapter. For $d=1$, Algorithm~\ref{mainalgo} can be written in a simplified form:
Let $\cS=(x_n)_{n\geq 1}$ be an infinite sequence in $[0,1)$. Then for every $N\geq 1$ we have $$ L_{2,N+1}(\cS)^2= L_{2,N}(\cS)^2+\sum_{n=1}^{N}x_n^2-2\sum_{n=1}^{N}\max\{x_n,x_{N+1}\}+(N+1)x_{N+1}^2-x_{N+1}+\frac{2N+1}{3}. $$
Clearly, $$ \argmin_{x_{N+1}\in [0,1)}L_{2,N}(\{x_1,\dots,x_{N},x_{N+1}\})^2=\argmin_{x\in [0,1)}f_N(x), $$
where $f_N: [0,1)\to \RR$ such that \begin{equation}
    \label{fndef}f_N(x):= -2\sum_{n=1}^{N}\max\{x_n,x\}+(N+1)x^2-x.
\end{equation} It is reasonable to choose $x_1$ such that the $L_2$ discrepancy of the one-element point set $\{x_1\}$ is minimal. This is the case for $x_1=\frac12$, as we see from equation~\eqref{minimal}, which leads to the following construction algorithm. 
\begin{algorithm} \label{algo1} We construct a sequence $\cS^*=(x_n)_{n\geq 1}$ in $[0,1)$ in the following way:
   \begin{enumerate}
	     \item Set $x_1=\frac12$. 
			 \item For $N\geq 1$: Assume that the elements $x_1,\dots,x_{N}$ are already constructed. Set $ x_{N+1}:=\min\argmin_{x\in [0,1)}f_N(x)$, where $f_N$ as defined in~\eqref{fndef}. 
	 \end{enumerate}
\end{algorithm}

\begin{figure}[ht] 
     \centering
     {\includegraphics[width=120mm]{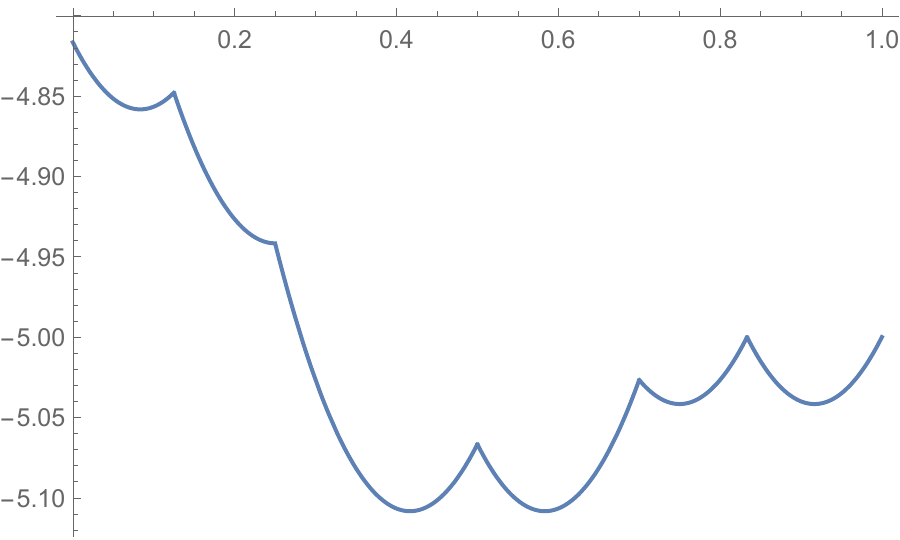}}
     \caption{The function $f_5$ takes global minima at $x=5/12$ and $x'=7/12$, where we pick the smaller argument $5/12$ in Algorithm~\ref{algo1}.}\label{emptyboxes}
\end{figure}

The fact that we choose the smallest element of $\argmin_{x\in [0,1)}f_N(x)$ is just a (random) choice to secure a unique output of the algorithm. We assume that the $L_2$ discrepancy of the resulting sequence is not affected significantly by doing so. We present the first 40 elements of the sequence $\cS^*$ generated by Algorithm~\ref{algo1}, which appears to be completely novel in the theory of uniform distribution modulo 1:
\begin{align*}
    \cS^*=&\left\{\frac12,\frac14,\frac56,\frac18,\frac{7}{10},\frac{5}{12},\frac{13}{14},\frac{1}{16},\frac{11}{18},\frac{7}{20},\frac{17}{22},\frac{5}{24},\frac{23}{26},\frac{13}{28},\frac{17}{30},\frac{1}{32},\frac{25}{34},\frac{11}{36},\frac{37}{38},\frac{7}{40},\right. \\ &\hspace{2mm}\left. \frac{9}{14},\frac{17}{44},\frac{37}{46},\frac{5}{48},\frac{27}{50},\frac{45}{52},\frac{5}{18},\frac{33}{56},\frac{9}{58},\frac{19}{20},\frac{27}{62},\frac{21}{64},\frac{15}{22},\frac{1}{68},\frac{53}{70},\frac{35}{72},\frac{67}{74},\frac{17}{76},\frac{49}{78},\frac{7}{80},...\right\}
\end{align*}
We observe that $x_N\in \Gamma_N$ for all $N\leq 40$. It is not difficult to show that this is the case for every $N$. To state the following two theorems, we introduce some notation. For a fixed $N\in\NN$ let $M_N:=\{x_1,\dots,x_N\}$ be the set of the first $N$ elements of the one-dimensional sequence $\cS^*$ generated by Algorithm~\ref{algo1}. Since these elements are pairwise distinct as we show in the following theorem, we may write $M_N=\{y_1,\dots,y_N\}$, where $y_1<y_2<\dots<y_N$. Hence, we sort the elements in $M_N$ and relabel them accordingly. Additionally, we set $y_0:=0$ and $y_{N+1}:=1$. For $k=1,\dots,N+1$ we define the functions $$\mathfrak{f}_k: \RR\to\RR; x\mapsto -(2k-1)x+(N+1)x^2-2\sum_{n=k}^{N}y_n.$$
  Then for all $k=1,\dots,N+1$ we clearly have \begin{equation} \label{restriction}
      f_N\bigg|_{x\in [y_{k-1},y_k]}=\mathfrak{f}_k\bigg|_{x\in [y_{k-1},y_k]}
  \end{equation} for the function $f_N$ as defined in~\eqref{fndef}. Finally we write $\Gamma_{N+1}=\{\gamma_1,\dots,\gamma_{N+1}\}$, where $\gamma_k:=\frac{2k-1}{2(N+1)}$ for $k=1,\dots,N+1$.

\begin{theorem} \label{mainl2}
Let $\cS^*$ be the sequence generated by Algorithm~\ref{algo1}. Then we have $x_N\in\Gamma_N$ for all $N\in \NN$. Further, $x_N$ is different from all previous elements of $\cS$.
\end{theorem}

\begin{proof}
  The assertion is obviously true for $N=1$. Let $N\geq 1$. From the induction hypothesis we have that the elements $\{x_1,\dots,x_{N}\}=\{y_1,\dots,y_{N}\}$ are pairwise distinct and nonzero. 
   Define the intervals $I_k=[y_{k-1},y_k]$ for $k\in \{1,\dots, N+1\}$. By $I_k'$ we denote the interior of $I_k$ for all $k\in \{1,\dots, N+1\}$. There clearly exists an index $l\in \{1,\dots,N+1\}$ such that $\gamma_l\in I_l'$, since $\gamma_k \notin I_k'$ for all $k\in\{1,\dots, N\}$ induces $\gamma_{N+1}\in I_{N+1}'$.
 For every $k=1,\dots,N+1$ the function $\mathfrak{f}_k(x)$ is differentiable on $I_k'$ and $\mathfrak{f}_k'(x)=-(2k-1)+2(N+1)x$. Since $\mathfrak{f}_k\big|_{x\in I_k}$ is a quadratic function defined on a closed interval and its graph is part of an upwardly open parabola, it can have its only global minimum either at $y_{k-1}$, at $y_k$ or at $\gamma_k$ in case that $\gamma_k\in I_k'$, since $\mathfrak{f}_k'(\gamma_k)=0$. To be more precise, for $k\in\{1,\dots,N+1\}$ the minimum of $\mathfrak{f}_k\big|_{x\in I_k}$ is at
 $$ \begin{cases}
    y_{k-1}, & \text{if \,} \gamma_k\leq y_{k-1}, \\
    \gamma_k, & \text{if \,} \gamma_k\in (y_{k-1}, y_k), \\
    y_k & \text{if \,} \gamma_k\geq y_{k}.
  \end{cases}$$
  For $k=1$ we can rule out the first case, whereas for $k=N+1$ the third case cannot occur.
  As a conclusion, we state that the arguments of the global minima of $f_N$ are elements of the set $\Gamma_{N+1} \cup \{x_1,\dots,x_{N}\}$. Now we will rule out the set $\{x_1,\dots,x_{N}\}$ as possible candidates for arguments of global minima. Assume that $\mathfrak{f}_k\big|_{x\in I_k}$ (for some $k\in\{2,\dots,N+1\}$) takes its minimum at $y_{k-1}$. Then $\gamma_{k-1}<\gamma_k\leq y_{k-1}$. Hence, $\mathfrak{f}_{k-1}\big|_{x\in I_{k-1}}$ takes its minimum either at $\gamma_{k-1}$ or at $y_{k-2}$ and $y_{k-1}$ cannot be the argument of a global minimum of $f_N$. If $\mathfrak{f}_k\big|_{x\in I_k}$ for some $k\in\{1,\dots,N\}$ takes its minimum at $y_{k}$, then $\gamma_{k+1}>\gamma_k\geq y_{k}$. Hence, $\mathfrak{f}_{k+1}\big|_{x\in I_{k+1}}$ takes its minimum either at $\gamma_{k+1}$ or at $y_{k+1}$ and $y_k$ cannot be the argument of a global minimum of $f_N$.  Hence, $f_N$ takes its global minimum at an element $\gamma_l\in\Gamma_{N+1}$ such that $\gamma_l\in I_l'$. Therefore, $\gamma_l$ is also different from all previous points of the sequence.
\end{proof}

Note that Theorem~\ref{mainl2} allows us to replace the command $  x_{N+1}:=\min\argmin_{x\in[0,1)}f_N(x)$ in Algorithm~\ref{algo1} by  $x_{N+1}:=\min\argmin_{x\in\Gamma_{N+1}}f_N(x)$, which makes it a lot faster. 

\begin{remark} \rm
Let us consider a modified version of Algorithm~\ref{algo1}, where we start the algorithm with $k\geq 2$ points. Choose an arbitrary set $\cP_k=\{x_1,\dots,x_k\}\subset [0,1]$ and for $N\geq k$ choose $x_{N+1}$ as in Algorithm~\ref{algo1}. Then with the very same arguments as in the proof of Theorem~\ref{mainl2} we can prove that $x_N\in \Gamma_N$ for all $N\geq k+1$. It is easy to see that a situation where elements of $\cP_k$ are equal does not cause any problems. The crucial observation in the proof of Theorem~\ref{mainl2} is that the function $f_N \big|_{x\in (y_{l-1},y_l)}$ is differentiable for every $l=1,2,\dots,N+1$ such that $y_{l-1}\neq y_l$ and that its derivative does not depend on the already generated points.  Therefore, regardless of which curious set of initial elements one likes to choose as input for the greedy algorithm, the subsequent elements $x_N$ for $N>k$ are all rational numbers of the form $x_N=\frac{2l-1}{2N}$ for some $l\in\{1,\dots,N\}$. This fact is remarkable, since usually similar greedy algorithms tend to produce numbers which can only be given numerically and which depend heavily on the first $k$ elements we provide as input for the algorithm.
\end{remark}

We would like to learn more about the structure of the sequence $\cS^*$. We prove that two consecutive elements of the first $N$ elements of the sequence generated by Algorithm 1 can never be too close to each other (which demonstrates that there cannot occur clusters in initial segments of the sequence).

\begin{theorem} \label{structure}
Let $\cS^*$ be the sequence generated by Algorithm 1 and let the first $N\geq 1$ elements $\{y_1,\dots,y_{N}\}$ already be generated.  Then  we have $\min_{0\leq k\leq N} (y_{k+1}-y_k)\geq \frac{1}{2N}$.
\end{theorem}

\begin{proof}
  The assertion is clearly true for $N=1$. Assume that it is true for $M_N:=\{x_1,\dots,x_{N}\}=\{y_1,\dots,y_{N}\}$. 
  It is easy to show that
  \begin{align}
      \mathfrak{f}_{k-1}(x)=&\mathfrak{f}_k\left(x+\frac{1}{N+1}\right)+2\left(\frac{k-1}{N+1}-y_{k-1}\right) \text{\, for \,} k=2,\dots,N+1, \text{\, and} \label{k} \\ \label{k+1}
      \mathfrak{f}_{k+1}(x)=&\mathfrak{f}_k\left(x-\frac{1}{N+1}\right)+2\left(y_k-\frac{k}{N+1}\right) \text{\, for \,} k=1,\dots,N.
  \end{align}
  Now assume that $x_{N+1}=\gamma_l=\frac{2l-1}{2(N+1)}$ for some $l\in\{1,\dots,N+1\}$. Then $\gamma_l\in (y_{l-1},y_l)$ and $f_N(\gamma_l)=-(N+1)\gamma_l^2-2\sum_{n=l}^N y_n$. We will show that $\left[\frac{l-1}{N+1},\frac{l}{N+1}\right)\cap M_N=\emptyset$, which implies the assertion of the theorem. Assume that $y_l \in \left(\gamma_l,\frac{l}{N+1}\right)$ (which is not possible if $l=N+1$, therefore we assume $l\leq N$ in the following). Since $\min_{0\leq k\leq N} (y_{k+1}-y_k)\geq \frac{1}{2N}$, there cannot be more than one element of $M_N$ lying in $\left(\gamma_l,\frac{l}{N+1}\right)$. Now we distinguish the two cases $y_{l+1}\leq \gamma_{l+1}=\frac{2l+1}{2(N+1)}$ and $y_{l+1}>\gamma_{l+1}$. In the latter case, we have 
  $$ f_N(\gamma_{l+1})=f_N(\gamma_l)+2\left(y_l-\frac{l}{N+1}\right)<f_N(\gamma_l) $$
  by equations~\eqref{restriction} and~\eqref{k+1}. This contradicts the fact that $f_N$ takes a global minimum at $x=\gamma_l$. If $y_{l+1}\leq\gamma_{l+1}$, then we can estimate
  \begin{align*}
    f_N(y_{l+1})=&\mathfrak{f}_{l+1}(y_{l+1})
            =-(2l+1)y_{l+1}+(N+1)y_{l+1}^2-2\sum_{n=l+1}^N y_n \\
            =&-(2l-1)y_{l+1}+(N+1)y_{l+1}^2-2(y_{l+1}-y_l)-2\sum_{n=l}^N y_n \\
            =& (N+1)\left(y_{l+1}-\gamma_l\right)^2-(N+1)\gamma_l^2-2(y_{l+1}-y_l)-2\sum_{n=l}^N y_n \\
            \leq& (N+1) \left(\frac{1}{N+1}\right)^2-\frac{1}{N}+f_N(\gamma_l)<f_N(\gamma_l),
  \end{align*}
  and again we get a contradiction. Therefore, $y_l \notin \left(\gamma_l,\frac{l}{N+1}\right)$. Using~\eqref{k} instead of~\eqref{k+1}, we conclude in the same fashion that also $y_{l-1} \notin \left(\frac{l-1}{N+1},\gamma_l\right)$. In remains to show that $y_{l-1}\neq \frac{l-1}{N+1}$ if $l\geq 2$. If we assume the opposite, i.e. $y_{l-1}= \frac{l-1}{N+1}$, then $y_{l-2}\leq \gamma_{l-1}$. Hence $f_N(\gamma_{l-1})=f_N(\gamma_l)$, and thus $\gamma_l$ is not the minimal argument of a global minimum of $f_N$ and we get another contradiction. The proof is complete.
\end{proof}

\begin{remark} \rm
We state further structural properties of the sequence $\cS^*$, which are only partially true.
\begin{enumerate}
    \item It follows from the proof of Theorem~\ref{structure} that $x_{N+1}$ occupies an interval of the form $\left[\frac{l-1}{N+1},\frac{l}{N+1}\right)$ for some $l=1,\dots,N+1$, where no previous element of $\cS^*$ already lies in. (Note that the same statement is true for some open interval of the form $\left(\frac{l-1}{N+1},\frac{l}{N+1}\right)$ if we do not necessarily pick the smallest argument of a global minimum of $f_N$. Therefore Theorem~\ref{structure} is also true in those cases.) Clearly, there is always at least one such empty interval $\left[\frac{l-1}{N+1},\frac{l}{N+1}\right)$, and for $1\leq N \leq 11$ it turns out that there is only one. However, this is not true in general, as for $N=12$ there already exist 2 intervals of the form $\left[\frac{l-1}{N+1},\frac{l}{N+1}\right)$ that do not contain any elements of $M_{12}$, namely for $l=9$ and $l=12$, whereas the interval $\left[\frac{10}{13},\frac{11}{13}\right)$ contains the two points $x_3=\frac56$ and $x_{11}=\frac{17}{22}$. Note that $x_{13}=\frac{23}{26}$ occupies the empty interval for the larger value $l=12$, since $-12.0302...=f_{12}(23/26)<f_{12}(17/26)=-12.0269...;$ hence the decision between these two points is very close and it appears difficult to determine in advance which empty interval will be occupied by $x_{N+1}$.
    \item Given a set of $N$ points $\cP=\{x_1,\dots,x_N\}$ in $[0,1)$ and a number $x\in [0,1)$. We define
$$ d(x,\cP):=\min_{k=1,2,\dots,N}|x-x_k|; $$
i.e. the distance of $x$ to its closest element of $\cP$. One might wonder whether the element $x_{N+1}$ of $\cS^*$ is always chosen as the minimal number $x\in \Gamma_{N+1}$ such that $d(x,M_N)=\max_{\gamma\in\Gamma_{N+1}} d(\gamma,M_N)$. Indeed, this seems to be the case for many $N$. Numerical calculations show that the assertion is true for all $N\in \{1,\dots,24\}\setminus\{14,15,16\}$. However, there are several exceptions from this rule and it cannot be used for an alternative construction algorithm of $\cS^*$. 
\item The initial segment of $\cS^*$ shows properties which are not true in general. For example, $x_{2^r}=\frac{1}{2^{r+1}}$ is true for $r=0,1,2,3,4$, $x_{3\cdot 2^r}=\frac{5}{3\cdot 2^{r+1}}$ holds for $r=0,1,2,3,4$ and $x_{5\cdot 2^r}=\frac{7}{5\cdot 2^{r+1}}$ for $r=0,1,2,3$, while $x_{2n}<\frac12$ and $x_{2n+1}\geq \frac12$ holds for $n\leq 12$ and more, but all these relations fail in general as verified by computing the first 1100 elements of $\cS^*$. As a consequence, it seems hard to come up with an explicit formula for the elements of $\cS^*$.

\end{enumerate}
\end{remark}

Theorem 1 yields the bound $L_{2,N}(\cS^*)\leq \sqrt{N/6}$ for all $N\geq 1$, which is a very bad upper bound compared to the best possible bound. Numerical experiments suggest that the $L_2$ discepancy of the sequence $\cS^*$ is much smaller than this bound. Figure 2 compares the $L_2$ discrepancy of the first $N$ elements of $\cS^*$ to the $L_2$ discrepancy of the first $N$ elements of the symmetrized van der Corput sequence up to $N=1100$, which indicates that $\cS^*$ has a lower $L_2$ discrepancy than $\widetilde{\cV}$ for most $N\geq 1$.

\begin{figure}[ht] 
     \centering
     {\includegraphics[width=150mm]{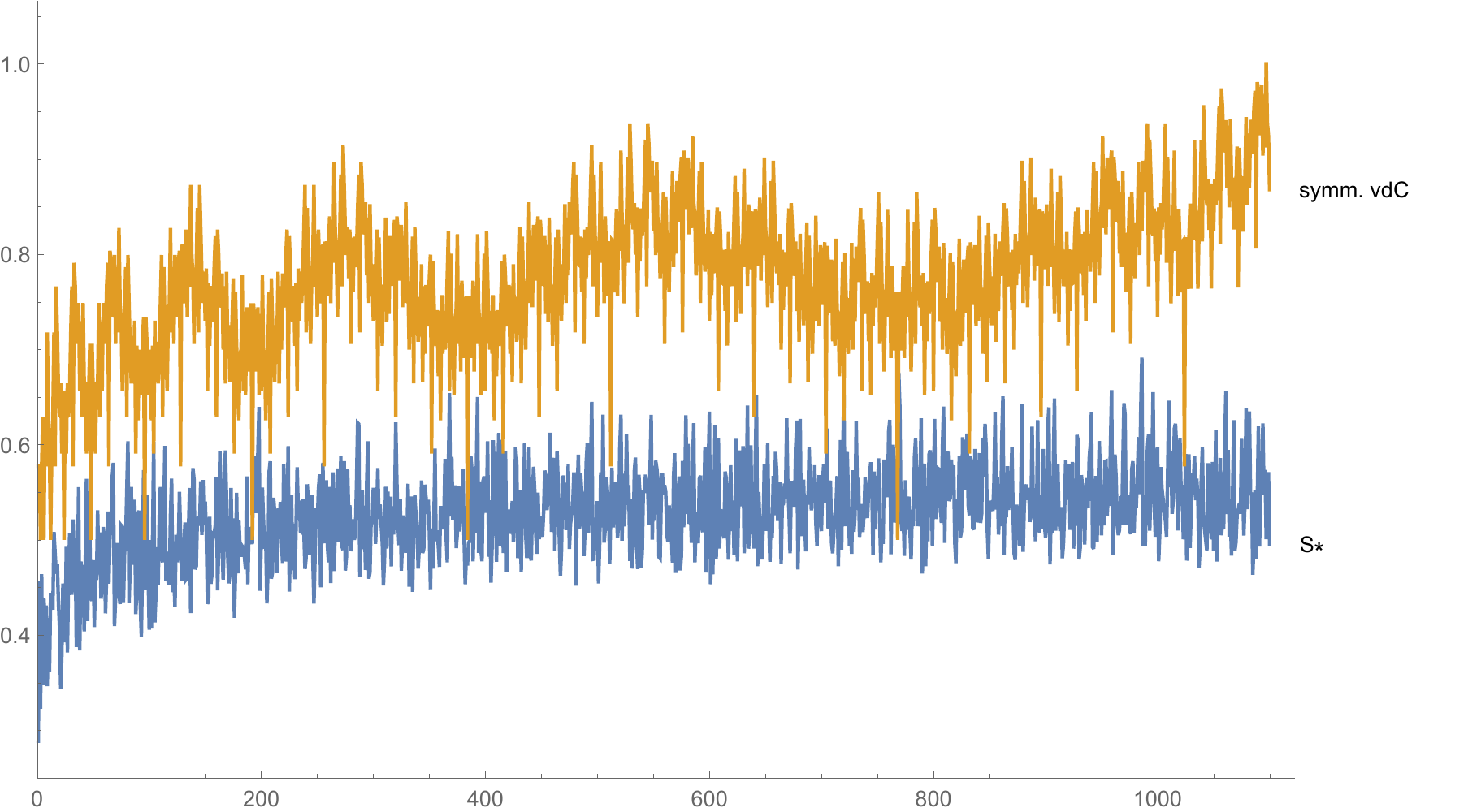}}
     \caption{For most $N$ we have $L_{2,N}(\cS^*)<L_{2,N}(\widetilde{\cV}$).}\label{compare}
\end{figure}

\begin{conjecture} \rm \label{conj1}
Let $\cS^*$ be the sequence generated by Algorithm 1. We conjecture that
  $$ \limsup_{N\to\infty}\frac{L_{2,N}(\cS^*)}{\sqrt{\log{N}}} <\limsup_{N\to\infty}\frac{L_{2,N}(\widetilde{\cV})}{\sqrt{\log{N}}}\leq 0.319553...$$
(see~\cite{fau90} for the second inequality). For a further hint towards this conjecture on the optimal $L_2$ discrepancy rate of $\cS^*$ we refer to the forthcoming Remark~\ref{vdcrem}. Every improvement of the bad discrepancy bound along with an explicit formula for the sequence $\cS^*$ would be desirable.
\end{conjecture}

We can also compare the star discrepancy of the first $N$ elements of $\cS^*$ to the star discrepancy of the first $N$ elements of the van der Corput sequence up to $N=1100$. In Figure~\ref{comparestar} we observe that the star discrepancy of $\cS^*$ is smaller for most $N$ and less fluctuating.

\begin{figure}[ht] 
     \centering
     {\includegraphics[width=130mm]{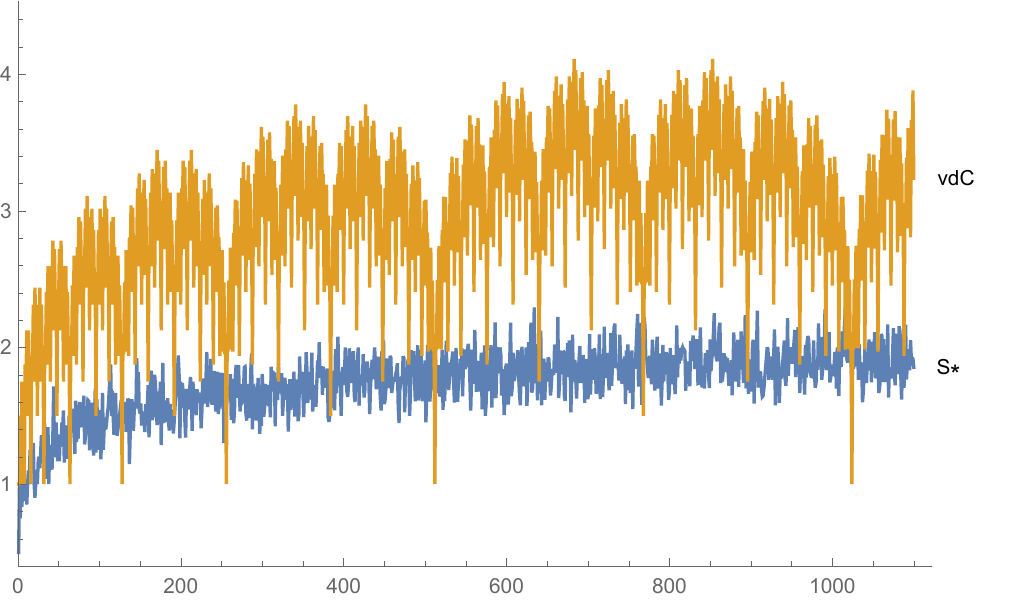}}
     \caption{The star discrepancy of the sequence $\cS^*$ is very small for $N\leq 1100$ which indicates that it is a low-discrepancy sequence.}\label{comparestar}
\end{figure}
\begin{conjecture} \rm \label{conj2}
We conjecture that the sequence $\cS^*$ is a low-discrepancy sequence; i.e. $D_N^*(\cS^*)=\mathcal{O}(\log{N})$. We even conjecture that
  $$ \limsup_{N\to\infty}\frac{D_N^*(\cS^*)}{\log{N}} <\limsup_{N\to\infty}\frac{D_N^*(\cV)}{\log{N}}=\frac{1}{3\log{2}} = 0.480898...$$
\end{conjecture}

For the sake of completeness, in Figure~\ref{comparestarkron} we compare the star discrepancy $D_N^*(\{n\alpha\})$ of the first $N$ elements of the $(n\alpha)$-sequence to $D_N^*(\cS^*)$ for $N\leq 1100$. We choose $\alpha=\sqrt{2}$, since by~\cite{dupain} it is known that
$$ \inf_{\alpha}\limsup_{N\to\infty}\frac{D_N^*(\{n\alpha\})}{\log{N}}=\limsup_{N\to\infty}\frac{D_N^*(\{n\sqrt{2}\})}{\log{N}}=\frac{1}{4\log{(\sqrt{2}+1)}}=0.283648...  $$ 

\begin{figure}[ht] 
     \centering
     {\includegraphics[width=130mm]{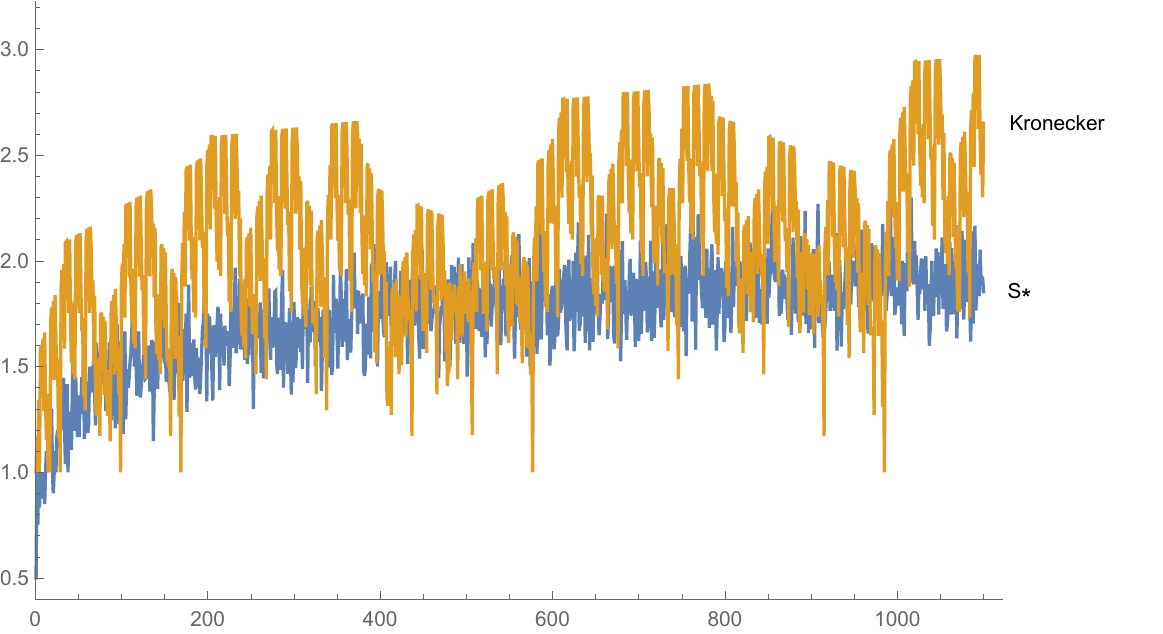}}
     \caption{The comparison between the star discrepancies of the two sequences suggests that $\cS^*$ even outperforms the Kronecker sequences with respect to star discrepancy asymptotically.}\label{comparestarkron}
\end{figure}

\subsection{Greedy minimization of the extreme $L_2$ discrepancy}

For $d=1$, the greedy algorithm as introduced in Theorem~\ref{maintheo} for $\bullet\in\{\mathrm{extr},\mathrm{per}\}$ has the following form. Here we make the particular choice $x_1=0$. Note that $L_{2,1}^{\mathrm{per}}(\{x_1\})^2=2 L_{2,1}^{\mathrm{extr}}(\{x_1\})^2=\frac16$ for every $x_1\in [0,1)$ and so the extreme and periodic $L_2$ discrepancy do not prefer any particular start values.

\begin{algorithm} \label{algo2}
   We construct a sequence $\cS'=(x_n)_{n\geq 1}$ in $[0,1)$ in the following way:
   \begin{enumerate}
	     \item Set $x_1=0$. 
			 \item For $N\geq 1$: Assume that the elements $x_1,\dots,x_{N}$ are already constructed. Set $ x_{N+1}:=\min\argmin_{x\in [0,1)}g_N(x)=\min\argmin_{x\in [0,1)}h_N(x),$
             where			 
			  $$ g_N(x):= -Nx(1-x)+2\sum_{n=1}^N (\min\{x_n,x\}-x_nx) $$ and 
			  $$ h_N(x):=\sum_{n=1}^N \left((x_n-x)^2-|x_n-x|\right). $$
	 \end{enumerate}
\end{algorithm}

Note that the functional $g_N$ stems from the greedy minimization of the extreme $L_2$ discrepancy, whereas $h_N$ comes from the periodic $L_2$ discrepancy. The equality $$\min\argmin_{x\in [0,1)}g_N(x)=\min\argmin_{x\in [0,1)}h_N(x)$$ is a direct consequence of the fact that $L_{2,N}^{\mathrm{per}}(\cS)^2=2 L_{2,N}^{\mathrm{extr}}(\cS)^2$ for every sequence in $[0,1)$. Alternatively, it is not difficult to show  directly that $g_N$ and $h_N$ share the same arguments of global minima by regarding the fact that $2\min\{x_n,x\}=x_n+x-|x_n-x|$. A beautiful result by Pausinger~\cite[Theorem 2.1]{Paus} immediately implies the following theorem.

\begin{theorem} \label{vdctheo}
Algorithm~\ref{algo2} generates the van-der-Corput sequence in base 2, i.e. $x_N=\varphi(N-1)$ for all $N\geq 1$.
\end{theorem}

\begin{proof} \label{mainex}
The sequence $\cS'$ is defined by $x_1=0$ and $$x_{N+1}=\min\argmin_{x\in[0,1)}h_N(x)=\min\argmin_{x\in[0,1)}\sum_{n=1}^{N}f(|x_n-x|)$$ with $f(x):=x^2-x.$ Since $f(x)=f(1-x)$ for all $x\in(0,1)$, $f$ is twice differentiable on $(0,1)$ and $f''(x)>0$ for all $x\in (0,1)$ we can apply \cite[Theorem 2.1]{Paus} and the result follows. Note that Pausinger's theorem tells us even more: if we do not always choose the smallest argument of a global minimum of $h_N$, Algorithm 2 still produces a generalized van der Corput sequence. We refer to \cite[Theorem 2.1]{Paus} for more details.
\end{proof}

We will give a direct proof of this result, based on Pausinger's ideas, in the Appendix.

\begin{remark}\rm \label{vdcrem}
\rm Theorem~\ref{maintheo} yields $L_{2,N}^{\mathrm{per}}(\cS')=\sqrt{2}L_{2,N}^{\mathrm{extr}}(\cS')\leq \sqrt{N/6}$ for all $N\in\NN$, whereas Theorem~\ref{vdctheo} implies the optimal upper bound $L_{2,N}^{\mathrm{per}}(\cS')=\sqrt{2}L_{2,N}^{\mathrm{extr}}(\cS')\lesssim \sqrt{\log{N}}$. We can understand this result as an indication that the upper bounds in Theorem~\ref{maintheo} are far from best possible and the $L_2$ discrepancy is much smaller than the given bounds. However, in order to prove the better bounds one needs to know the structure of the sequences resulting from the greedy algorithms, which is probably a very difficult task to investigate; especially in dimension $d\geq 2$. Further, due the fact that Algorithm 2 leads to a sequence with the optimal order of extreme and periodic $L_2$ discrepancy, we can conjecture that Algorithm~\ref{algo1} generates a sequence in $[0,1)$ with the optimal order of star $L_2$ discrepancy; i.e. $L_{2,N}(\cS^*)\lesssim \sqrt{\log{N}}$. From this point of view it is clear why Algorithm~\ref{algo1} does not generate the van der Corput sequence which fails to have the optimal order of $L_2$ discrepancy.
\end{remark}

\begin{remark} \rm
Algorithm~\ref{algo2} does not have the nice property of Algorithm~\ref{algo1} that the resulting elements of the sequence are all rational numbers independently of the set of points we start the algorithm with, because the derivative of $g_N\big|_{x\in (y_{l-1},y_l)}$ depends on the already constructed points for $l\in\{2,\dots,N+1\}$.
\end{remark}

\section{Tractability} \label{tract}

In this section we would like to mention connections of our results to the theory of information based complexity. In this field, one is for instance interested in the dependence of the discrepancy on the dimension $d$. In this section we would like to know the minimal number $N(d,\varepsilon)$ such that there exists an $N$-element point set $\cP$ in $[0,1)^d$ such that its normalized $L_2$ discrepancy $N^{-1}L_{2,N}(\cP)$ is smaller than a given $\varepsilon>0$. If $N(d,\varepsilon)$ depends only polynomially on $d$ and on $\varepsilon^{-1}$, one speaks of polynomial tractability. If there is no dependence on the dimension $d$ at all, one speaks of strong polynomial tractability. 
\\
On average, the squared normalized $L_2$ discrepancy of an $N$-element point set is 
\begin{equation} \label{average} \int_{[0,1]^{Nd}} \left(N^{-1}L_{2,N}(\{\bsx_1,\dots,\bsx_N\})\right)^2 \rd\bsx_1\cdots\rd\bsx_N=\left(\frac{1}{2^d}-\frac{1}{3^d}\right)N^{-1}=c_d^{\mathrm{star}}N^{-1} \end{equation}
(see e.g.~\cite{moro}), which proves the existence of an $N$-element point set in $[0,1)^d$ such that its normalized $L_2$ discrepancy is bounded by $\sqrt{c_d^{\mathrm{star}}}/\sqrt{N}$ . Note that Algorithm~\ref{mainalgo} generates an infinite sequence where the segment $\{\bsx_1,\dots,\bsx_N\}$ of its first $N$ elements satisfies this bound for every $N\geq 1$ (provided that we start the algorithm with a single element $\bsx_1\in [0,1]^d$ such that $L_{2,1}(\{\bsx_1\})\leq \sqrt{c_d^{\mathrm{star}}}$). Similar statements are true for the extreme and periodic $L_2$ discrepancy.\\
The dependence on $d$ in~\eqref{average} seems to be very good, as $c_d^{\mathrm{star}}$ decreases exponentially in $d$. However, this positive interpretation of the formula is deceptive. As pointed out by Novak and Wo\'zniakowski in~\cite{nw1}, the $L_2$ discrepancy is not properly normalized in the following sense: the normalized $L_2$ discrepancy of the empty set is given by
$$ \left(\int_{[0,1]^d} |0-|[\bszero,\bst)||^2\rd\bst\right)^{\frac12}=\left(\int_{[0,1]^d} t_1^2\cdot\dots\cdot t_d^2\rd\bst\right)^{\frac12}=3^{-d/2}. $$
Hence, it makes more sense to consider the quantity
$$ N_2(d,\varepsilon):=\min \{N\in\NN: \inf_{\cP\in [0,1]^d,\, \#\cP=N}L_{2,N}(\cP)<\varepsilon3^{-d/2}\}, $$
which is often referred to as the inverse of the $L_2$ discrepancy.
 Since $\sqrt{c_d^{\mathrm{star}}}/\sqrt{N}<\varepsilon3^{-d/2}$ is equivalent to $N>\left(1.5^d-1\right)\varepsilon^{-2}$, it is impossible to derive polynomial tractability from the discrepancy bound in~\eqref{average} with respect to the normalized error criterion $N_2(d,\varepsilon)$. It is known (see e.g.~\cite[Proposition 3.58]{DP}) that the normalized $L_2$ discrepancy is not tractable at all; in particular we have
 $$ N_2(d,\varepsilon)\geq \left(\frac98\right)^d (1-\varepsilon^{2}). $$
 To overcome the bad dependency on $d$, Novak and Wo\'zniakowski~\cite{nw1} introduced the concept of weighted $L_2$ discrepancy, where groups of variables are weighted differently. For every subset $\mathfrak{u}\subseteq \{1,\dots,d\}$ one chooses a positive weight $\gamma_{\mathfrak{u}}$. For the definition of the weighted $L_2$ discrepancy we refer to~\cite{DP,joe,nw1,NW10}. Here, we only remark that the choice $\gamma_{\{1,\dots,d\}}=1$ and $\gamma_{\mathfrak{u}}=0$ for all $\mathfrak{u}\neq \{1,\dots,d\}$ yields the standard star $L_2$ discrepancy. A popular choice of weights are product weights, where one selects a non-increasing sequence $\bsgamma=(\gamma_j)_{j\geq 1}$ of positive weights and puts
 $$ \gamma_{\mathfrak{u}}=\prod_{j\in \mathfrak{u}}\gamma_j. $$
 If one chooses product weights, the weighted $L_2$ discrepancy $L_{2,N}^{\bsgamma}(\cS)$ of the first $N$ elements of a sequence $\cS=(\bsx_n)_{n\geq 0}$ with $\bsx_n=(x_{n,1},\dots,x_{n,d})$ for all $n\in\NN$ is given explicitly by the Warnock-type formula (see~\cite[Theorem 2.1]{joe})
 \begin{align*}
     L_{2,N}^{\bsgamma}(\cS)=&N^2\prod_{i=1}^d \left(1+\frac{\gamma_i}{3}\right)-2N\sum_{n=1}^N\prod_{i=1}^d \left(1+\frac{\gamma_i}{2}(1-x_{n,i}^2)\right) \\
     &+\sum_{n,m=1}^N \prod_{i=1}^d \left(1+\gamma_i(1-\max\{x_{n,i},x_{m,i}\})\right).
 \end{align*} 
 We formulate a greedy algorithm in the style of Algorithm~\ref{mainalgo} which is based on the weighted $L_2$ discrepancy, and state a corresponding discrepancy result in the following theorem. The proof is similar to the proof of Theorem~\ref{maintheo}.
\begin{theorem}
  Let $\cS_d^{\bsgamma}=\{\bsx_1,\bsx_2,...\}\subset [0,1)^d$ be generated as follows:
\begin{enumerate} \label{mainalgo2}
    \item Choose $\bsx_1\in [0,1]^d$ such that $L_{2,1}^{\bsgamma}(\bsx_1)^2\leq c_d^{\bsgamma}$, where
    $$ c_d^{\bsgamma}:=\prod_{i=1}^d \left(1+\frac{\gamma_i}{2}\right)-\prod_{i=1}^d \left(1+\frac{\gamma_i}{3}\right). $$
    \item For $N\geq 1$ let $\{\bsx_1,\dots,\bsx_{N}\}$ already be given.      Choose
         \begin{equation} 
\bsx_{N+1}\in\argmin_{\bsy\in[0,1)^d}L_{2,N+1}^{\bsgamma}(\{\bsx_1,\dots,\bsx_{N},\bsy\})=\argmin_{\bsy\in[0,1)^d}f_{N,d}^{\bsgamma}(\bsy),
         \end{equation}
        where $f_{N,d}^{\bsgamma}: [0,1]^d\to \RR$ is defined as
          \begin{align*}
             f_{N,d}^{\bsgamma}(\bsy)=&-2(N+1)\prod_{i=1}^d \left(1+\frac{\gamma_i}{2}(1-y_i^2)\right)+2\sum_{n=1}^N \prod_{i=1}^d \left(1+\gamma_i(1-\max\{x_{n,i},y_i\})\right) \\
             &+\prod_{i=1}^d \left(1+\gamma_i (1-y_i)\right)
          \end{align*}
        for $\bsy=(y_1,\dots,y_d)$. 
    \end{enumerate}Then $L_{2,N}^{\bsgamma}(\cS_d^{\bsgamma})^2\leq c_d^{\bsgamma}N$ for all $N\geq 1$.
\end{theorem}
Note that it is possible to find a point $\bsx_1\in [0,1]^d$ such that $L_{2,1}(\bsx_1)^2\leq c_d^{\bsgamma}$, since the average of the squared weighted $L_2$ discrepancy over all $N$-element point sets in $[0,1]^d$ is exactly $c_d^{\bsgamma}N$ (see e.g.~\cite[Section 3.6]{DP}). Since
$$ c_d^{\bsgamma}<\prod_{i=1}^d \left(1+\frac{\gamma_i}{2}\right)=\mathrm{e}^{\sum_{i=1}^d \log{\left(1+\frac{\gamma_i}{2}\right)}}<\mathrm{e}^{\frac12\sum_{i=1}^d \gamma_i}; $$
hence we have $L_{2,N}^{\bsgamma}(\cS_d^{\bsgamma})^2\leq c_{\bsgamma} N$ with a positive constant $c_{\bsgamma}$ independent of $d$ if $$\sum_{i=1}^{\infty} \gamma_i <\infty.$$ We have $L_{2,N}^{\bsgamma}(\cS_d^{\bsgamma})^2\leq c_{\bsgamma}' d^a N $ with positive constants $c_{\bsgamma}'$ and $a$ independent of $d$ if
$$ \limsup_{d\to\infty}\frac{\sum_{i=1}^d \gamma_i}{\log{d}}<\infty. $$
Therefore, under suitable conditions on the weights we achieve (strong) polynomial tractability. Note that under the same conditions on the weights we obtain the same tractability notions for the normalized error criterion; i.e. the situation where we demand $L_{2,N}^{\bsgamma}(\cS_d^{\bsgamma})\leq \varepsilon \prod_{i=1}^d \left(1+\frac{\gamma_i}{3}\right)^{1/2}$, where $\prod_{i=1}^d\left(1+\frac{\gamma_i}{3}\right)^{1/2}$ is the normalized weighted $L_2$ discrepancy of the empty set (see~\cite[Section 7]{nw1} or~\cite[Section 9.3]{NW10}).

\section{Conclusion} \label{sec5}

We considered greedy algorithms where we choose $k\geq 1$ elements in $[0,1)^d$; i.e. an initial set of points $\{\bsx_1,\dots,\bsx_k\}\subset [0,1)^d$. The point $\bsx_{k+1}$ is then chosen such that a certain variant of $L_2$ discrepancy of the point set $\{\bsx_1,\dots,\bsx_k,\bsx_{k+1}\}$ is minimized. All subsequent elements of the sequence $\cS_d=\{\bsx_n\}_{n\geq 1}$ are selected in the same way. We proved that all sequences we can generate with this method are uniformly distributed modulo 1 and satisfy $L_{2,N}^{\bullet}(\cS_d)\leq C_d\sqrt{N}$ for a suitable notion of $L_2$ discrepancy, where the positive constant $C_d$ depends only on the dimension $d$ and on the initial set of points $\{\bsx_1,\dots,\bsx_k\}$. \\
We proved precise results on the resulting sequences in the one-dimensional case, where we put most attention on cases where we start the algorithms with a single element $x_1\in [0,1)$. A greedy minimization of the star $L_2$ discrepancy yields a natural extension of any initial segment $\cP_k:=\{x_1,\dots,x_k\}\subset [0,1)$ to a uniformly distributed sequence, where $x_N=\frac{2n-1}{2N}$ with some $n\in\{1,\dots,N\}$ for all $N\geq k+1$. We analysed the situation where $\cP_1=\{\frac12\}$ in more detail, where already for the first element the star $L_2$ discrepancy is minimized. We proved that two consecutive elements in the first $N$ elements of this sequence must always have a distance of at least $\frac{1}{2N}$. We also found numerically that the resulting sequence is likely to be low-discrepancy and might have significantly lower star discrepancy than the classical van der Corput sequence. 

If we consider an algorithm based on a greedy minimization of the one-dimensional extreme or periodic $L_2$ discrepancy, a result by Pausinger immediately yields that for the initial set $\cP_1=\{0\}$ we obtain the van der Corput sequence or a permuted variant thereof. Therefore, this greedy algorithm indeed generates a low-discrepancy sequence. However, the situation for general $\cP_k$ is less clear than in case of a minimization of the star $L_2$ discrepancy, since a greedy algorithm based on the extreme $L_2$ discrepancy does not produce all rational points in general.

Besides the rough discrepancy bound $L_{2,N}^{\bullet}(\cS_d)\leq C_d\sqrt{N}$ we were not yet able to prove better and more precise results  on higher-dimensional sequences generated by our greedy algorithms in the order of $N$. We assume that they might be low-discrepancy as well. However, if we consider a greedy algorithm based on the weighted $L_2$ discrepancy, under certain conditions on the weights the resulting sequence satisfies excellent $L_2$ discrepancy bounds with respect to the dimension $d$ for all $N\geq 1$ simultaneously.

Summarizing, there remain many interesting open problems related to the present work, where we would like to point out Conjectures~\ref{conj1} and~\ref{conj2} from Section 3 in particular.

\section*{Appendix - A direct proof of Theorem~\ref{vdctheo}}

Although the proof of Theorem~\ref{vdctheo} is complete by simply citing Pausinger's result, we would like to give a detailed and direct proof based on his arguments to give insight why exactly the van der Corput sequence is the output of Algorithm~\ref{algo2}. For a fixed $N\in\NN$ we introduce the function $G_N: [0,1)\to\RR$ such that
$$ G_N(x):=-Nx(1-x)+2\sum_{n=0}^{N-1} (\min\{\varphi(n),x\}-\varphi(n)x). $$
It is clear that we have to show that for all $N\geq 1$ we have
\begin{equation}
    \label{toshow}\min\argmin_{x\in [0,1)}G_N(x)=\varphi(N)
\end{equation}
in order to prove Theorem 2. We show~\eqref{toshow} in Corollary~\ref{power2}, Proposition~\ref{odds} and Corollary~\ref{generals}, which will conclude the proof. First, we prove two crucial properties of the function $G_N$:

\begin{lemma} \label{gprop}
Let be $N\geq 1$ and $r\in\NN_0$ maximal such that $2^r$ divides $N$; i.e. $N=2^r m$ for some odd integer $m\geq 1$.
\begin{enumerate}
\item Then $G_N$ is $2^{-r}$-periodic; i.e. for every $x\in [0,2^{-r})$ we have $G_N(x+2^{-r}l)=G_N(x)$ for all $l\in\{0,1,\dots,2^r-1\}$.
\item We have $G_m(2^rx)=2^rG_N(x)$ for all $x\in [0,2^{-r})$.
\end{enumerate}
\end{lemma}

\begin{proof}
Consider the set $M_N:=\{\varphi(0),\varphi(1),\dots,\varphi(N-1)\}=\{y_1,y_2,\dots,y_N\}$, where $y_1<y_2<\dots<y_N$. Let $y_i \in [0,2^{-r})$; then $y_i=\varphi(2^{r}t)$ for some $t\in\{0,1,\dots,m-1\}$. Then we have
\begin{align*}
     \left\{\varphi(2^{r}t+s): s\in\{0,1,\dots,2^r-1\}\right\}=&\left\{\varphi(2^{r}t)+\varphi(s): s\in\{0,1,\dots,2^r-1\}\right\} \\
     =&\left\{\varphi(2^{r}t)+\frac{w}{2^r}: w\in\{0,1,\dots,2^r-1\}\right\}, 
\end{align*}
which implies that for every element $y_i\in M_N\cap [0,2^{-r})$ also $y_i+2^{-r}w\in M_N$ for all $w\in \{0,1,\dots,2^r-1\}$ and that every element of $M_N$ can be expressed that way. We use this fact to show that
\begin{align*}
    \sum_{n=0}^{N-1}\varphi(n)=&\sum_{w=0}^{2^r-1}\sum_{n=1}^{m}\left(y_n+2^{-r}w\right)=2^r\sum_{n=1}^{m}y_n+m\left(2^{r-1}-\frac12\right)
\end{align*}
and therefore
\begin{equation} \label{gug}
    \sum_{n=1}^{m}y_n=2^{-r}\left(\sum_{n=0}^{N-1}\varphi(n)-m\left(2^{r-1}-\frac12\right)\right).
\end{equation}
Now choose an arbitrary $x\in [0,2^{-r})$. Let $k$ be maximal such that $y_k \leq x$. Then we have
$$ \sum_{n=0}^{N-1} \min\{\varphi(n),x\}=\sum_{n=1}^k y_n+(N-k)x $$
and
\begin{align*}
   \sum_{n=0}^{N-1}& \min\{\varphi(n),x+2^{-r}l\} \\
     =&\sum_{w=0}^{l-1}\sum_{n=1}^{m} (y_n+2^{-r}w)+\sum_{n=1}^k (y_n+2^{-r}l)+(N-lm-k)(x+2^{-r}l) \\
     =&l\sum_{n=1}^{m} y_n+m2^{-r-1}l(l-1)+2^{-r}kl+(N-k)2^{-r}l-lm(x+2^{-r}l) \\
     &+\sum_{n=0}^{N-1} \min\{\varphi(n),x\} \\
     =&2^{-r}l\sum_{n=0}^{N-1}\varphi(N)-\frac{lm}{2}+N2^{-r}l-lmx-2^{-r-1}l^2m+\sum_{n=0}^{N-1} \min\{\varphi(n),x\}, 
\end{align*}
where we used~\eqref{gug} in the last line. From the last line follows immediately
\begin{align*}
    G_N(x+2^{-r}l)=&-N(x+2^{-r}l)(1-x-2^{-r}l)\\
     &+2\sum_{n=0}^{N-1} (\min\{\varphi(n),x+2^{-r}l\}-\varphi(n)(x+2^{-r}l))=G_N(x),
\end{align*}
as all terms that do not belong to $G_N(x)$ cancel out. \\
We prove the second item. With the arguments and notations from above we have
$$ \sum_{n=0}^{m-1}\min\{\varphi(n),2^rx\}=\sum_{n=1}^{k}(2^ry_n)+(m-k)2^r x=2^r \left(\sum_{n=0}^{N-1}\min\{\varphi(n),x\}+(m-N)x\right)$$
and
$$ \sum_{n=0}^{m-1}\varphi(n)=\sum_{n=1}^{m}(2^ry_n)=\sum_{n=0}^{N-1}\varphi(n)-m\left(2^{r-1}-\frac12\right) $$
by equation~\eqref{gug}. That yields
\begin{align*}
    G_m(2^rx)=& -m2^rx(1-2^rx)+2\sum_{n=0}^{m-1}\min\{\varphi(n),2^rx\}-2^{r+1}x\sum_{n=0}^{m-1}\varphi(n) \\
    =& 2^r \left(G_N(x)+Nx(1-x)-mx(1-2^rx)+2(m-N)x+2xm\left(2^{r-1}-\frac12\right)\right) \\
    =& 2^r G_N(x),
\end{align*}
and the proof is complete.
\end{proof}

We immediately conclude

\begin{corollary} \label{power2}
For all $r\geq 0$ we have $\argmin_{x\in [0,1)}G_{2^r}(x)=\Gamma_{2^r}.$
\end{corollary}

\begin{proof}
We use Lemma~\ref{gprop} to obtain $ G_{2^r}(x)=2^{-r}G_1(2^rx)=-x(1-2^rx)$ for all $x\in [0,2^{-r})$. Hence, $x^*=2^{-r-1}$ is the argument of the only global minimum of $G_N$ in $[0,2^{-r})$. The rest follows by the periodicity property of $G_N$ given in Lemma~\ref{gprop}.
\end{proof}

We note that $ \min\argmin_{x\in [0,1)}G_{2^r}(x)=2^{-r-1}=\varphi(2^r) $, which proves~\eqref{toshow} for powers of 2. Next, we show the result for odd integers $N\geq 1$. For $N=1$ it is obvious. In the next proof, we employ the reasoning of Pausinger.

\begin{proposition} \label{odds}
For an odd integer $N\geq 3$ we have~\eqref{toshow}.
\end{proposition}

\begin{proof}
 Write $N=\sum_{j=1}^k 2^{m_j}$, where $m_k>m_{k-1}>\dots>m_1=0$ are integers. Set $N_i:=\sum_{j=i}^k 2^{m_j}$ for $i=1,\dots,k$ and $N_{k+1}:=0$.
 
  Now we can write $G_N=\sum_{i=1}^{k}\widetilde{G}_{2^{m_i}}$,
  where we set 
  \begin{align*}
      \widetilde{G}_{2^{m_i}}(x):=&-2^{m_i}x(1-x)+2\sum_{n=N_{i+1}}^{N_i-1} (\min\{\varphi(n),x\}-\varphi(n)x) \\
      =&-2^{m_i}x(1-x)+2\sum_{n=0}^{2^{m_i}-1} (\min\{\varphi(n+N_{i+1}),x\}-\varphi(n+N_{i+1})x) \\
      =&G_{2^{m_i}}(x-\varphi(N_{i+1}))+\underbrace{2^{m_i}\varphi(N_{i+1})(1-\varphi(N_{i+1}))-2\varphi(N_{i+1})\sum_{n=0}^{2^{m_i}-1}\varphi(n)}_{\text{independent of $x$}},
  \end{align*} 
  where in the last step we regarded $\varphi(n+N_{i+1})=\varphi(n)+\varphi(N_{i+1})$ and applied some elementary algebra.
  From the last equality together with Corollary~\ref{power2} we conclude that 
  \begin{align*} \argmin_{x\in [0,1)} \widetilde{G}_{2^{m_i}}(x) =&\Gamma_{2^{m_i}}+\varphi(N_{i+1}) \\ =&\left\{\frac{l}{2^{m_i}}+\frac{1}{2^{m_i+1}}+\sum_{j=i+1}^{k}\frac{1}{2^{m_j+1}}: l=0,1,\dots,2^{m_i}-1\right\}. \end{align*}
  It is now obvious that we have
  $$ \argmin_{x\in [0,1)} \widetilde{G}_{2^{m_i}}(x)\subset \argmin_{x\in [0,1)} \widetilde{G}_{2^{m_{i+1}}}(x) $$
  for all $i=1,\dots,k-1$, which implies
  \begin{align*} \argmin_{x\in [0,1)}G_N(x)=&\bigcap_{i=1}^k \argmin_{x\in [0,1)} \widetilde{G}_{2^{m_i}}(x)=\argmin_{x\in [0,1)} \widetilde{G}_{2^{m_1}}(x)\\ =&\left\{\frac{1}{2^{m_1+1}}+\sum_{j=2}^{k}\frac{1}{2^{m_j+1}}\right\}=\{\varphi(N)\}. \end{align*}
  Hence, $\min \argmin_{x\in [0,1)}G_N(x)=\varphi(N)$ as claimed.
\end{proof}

\begin{corollary} \label{generals}
Let $N=2^r m$ with an integer $r\geq 1$ and an odd integer $m\geq 3$. Then for $N$ we have~\eqref{toshow}.
\end{corollary}

\begin{proof}
 Since $G_N(x)=2^{-r}G_m(2^r x)$ for $x\in [0,2^{-r})$ by Lemma~\ref{gprop}, we derive from Proposition~\ref{odds} that $G_N(x)$ takes a unique global minimum in $[0,2^{-r})$ at $x^*=\varphi(m)/2^r=\varphi(N)$, which proves the corollary.
\end{proof}

\paragraph{Acknowledgements} I would like to thank Stefan Steinerberger for his friendly feedback and his  suggestions on how to highlight the most interesting aspects of the results more clearly. I also thank an anonymous referee for very good suggestions on how to rearrange the paper in order to improve its readability and how to avoid unnecessary abbreviations and technicalities in the introduction.

\noindent{\bf Author's Address:}\\
Ralph Kritzinger, Leopold-Werndl-Stra\ss e 25a, A-4400 Steyr, Austria.\\
Email: 
ralph.kritzinger@yahoo.de.

\end{document}